\documentclass{article}%
\usepackage{amsmath}
\usepackage{amsfonts}
\usepackage{amssymb}
\usepackage{graphicx}%
\providecommand{\U}[1]{\protect\rule{.1in}{.1in}}
\newtheorem{theorem}{Theorem}
\newtheorem{acknowledgement}[theorem]{Acknowledgement}

\newtheorem{corollary}[theorem]{Corollary}

\newtheorem{lemma}[theorem]{Lemma}

\newtheorem{proposition}[theorem]{Proposition}
\newtheorem{remark}[theorem]{Remark}

\newenvironment{proof}[1][Proof]{\noindent\textbf{#1.} }{\ \rule{0.5em}{0.5em}}
\begin{document}

\title{Three-dimensional noncompact $\kappa$-solutions that are Type I forward and backward}
\author{Xiaodong Cao\\Department of Mathematics, Cornell University, Ithaca, NY 14853\\\emph{E-mail address}: \texttt{cao@math.cornell.edu}
\and Bennett Chow\\Department of Mathematics, UC-San Diego, La Jolla, CA 92093\\\emph{E-mail address}: \texttt{chowbennett@gmail.com}
\and Yongjia Zhang\\Department of Mathematics, UC-San Diego, La Jolla, CA 92093\\\emph{E-mail address}: \texttt{yoz020@ucsd.edu}}
\maketitle

As indicated by the third author in \cite{Zhang}, there is a gap in the previous version of this paper by the first two authors \cite{CaoChow}. We provide in this version an argument to fix the aforementioned gap. The main proposition, whose proof uses Perelman's techniques, is implied by Ding \cite{Ding} and is covered by \cite{Zhang}. Our approach, however, is different from theirs. In addition, we prove a necessary and sufficient condition for a three-dimensional $\kappa$-solution to form a forward singularity. We hope that this condition is helpful in the classification of all three-dimensional $\kappa$-solutions. Up to now, the only main progress on such a classification, as conjectured by Perelman, is by Brendle \cite{Brendle2013}.

\section{Introduction}

A complete solution to the backward Ricci flow $(\mathcal{M}^{n},g(\tau))$,
$\tau\in(0,\infty)$, is a \textbf{Type I }$\kappa$\textbf{-solution} if
$\left\vert \operatorname{Rm}\right\vert (\tau)\leq\frac{C}{\tau}$ for some
constant $C$ and each $g(\tau)$ is $\kappa$-noncollapsed below all scales. In
this definition we do not assume nonnegativity of the curvatures.

As a special case of his result Proposition 0.1 in \cite{NiClosedTypeI}, Lei
Ni has classified $3$-dimensional closed Type I $\kappa$-solutions. In all
dimensions Ni first showed that if $(\mathcal{M}^{n},g(\tau))$, $\tau
\in(0,\infty)$, is a solution to the backward Ricci flow on a compact manifold
and satisfies $\left\vert \operatorname{Rm}\right\vert (\tau)\leq\frac{A}%
{\tau}$, then there exists a constant $C(n,A)$ such that%
\begin{equation}
\operatorname{diam}\left(  g(\tau)\right)  \leq\max\{\operatorname{diam}%
\left(  g(1)\right)  ,C(n,A)\}\sqrt{\tau}\quad\text{for }\tau\in
(1,\infty).\label{LeiNi1}%
\end{equation}
In particular, by Lemma 8.3(b) in Perelman \cite{Perelman1}, there exists
$C=C(n,A)$ such that $\frac{\partial}{\partial\tau}d_{\tau}(x_{1},x_{2}%
)\leq\frac{C(n,A)}{2\sqrt{\tau}}$ for any $x_{1},x_{2}\in\mathcal{M}$ with
$d_{\tau}(x_{1},x_{2})\geq C(n,A)\sqrt{\tau}$. Then (\ref{LeiNi1}) follows
from the consequence that for $x_{1},x_{2}\in\mathcal{M}$ and $\tau\geq1$ we
have%
\[
\frac{d_{\tau}(x_{1},x_{2})}{\sqrt{\tau}}<C(n,A)\quad\text{or}\quad
\frac{\partial}{\partial\tau}\frac{d_{\tau}(x_{1},x_{2})}{\sqrt{\tau}}\leq0.
\]

Using this, Ni proved that if $\left(  \mathcal{M}^{n},g\left(  \tau\right)
\right)  $, $\tau\in\left(  0,\infty\right)  $, is a closed Type I $\kappa
$-solution with positive curvature operator (PCO), then $\left(
\mathcal{M},g(\tau)\right)  $ is isometric to a shrinking spherical space
form. In particular, since $g\left(  \tau\right)  $ has PCO and $\mathcal{M}$
is closed, by Hamilton \cite{Hamilton1982}, \cite{Hamilton1986} when $n=3,4$
and B\"{o}hm and Wilking \cite{BohmWilking} when $n\geq5$, $g\left(
\tau\right)  $ converges to a constant positive sectional curvature (CPSC)
metric $g_{0}$ as $\tau\rightarrow0$. By 11.2 in \cite{Perelman1}, fixing $p$,
there exist $q_{i}$ such that $\ell_{(p,0)}^{g}\left(  q_{i},i\right)
\leq\frac{n}{2}$ and $\left(  \mathcal{M}^{n},i^{-1}g\left(  i\tau\right)
,q_{i}\right)  $ subconverges in the Cheeger--Gromov sense to a complete
nonflat shrinking gradient Ricci soliton (GRS) $\left(  \mathcal{M}_{\infty}^{n},g_{\infty}\left(
\tau\right)  ,q_{\infty}\right)  $. By (\ref{LeiNi1}), we have%
\[
\operatorname{diam}\left(  \frac{1}{i}g\left(  i\tau\right)  \right)  \leq
\max\{\operatorname{diam}\left(  g(-1)\right)  ,C(n,A)\}\sqrt{\tau}%
\quad\text{for }\tau\in(\frac{1}{i},\infty).
\]
Thus $\mathcal{M}_{\infty}$ is compact and diffeomorphic to $\mathcal{M}$.
Since $(\mathcal{M},g_{\infty}\left(  \tau\right)  )$ is irreducible with
nonnegative curvature operator on a topological spherical space form,
$g_{\infty}\left(  \tau\right)  $ must be a CPSC metric. By all of the above,
after rescaling, $g\left(  \tau\right)  $ converges to a metric which is
isometric to a constant multiple of $g_{\infty}\left(  1\right)  $ as either
$\tau\rightarrow0$ or $\tau\rightarrow\infty$. This implies that Perelman's
invariant $\nu(g(\tau))$ must be constant, which implies that $g(\tau)$ is a
shrinking GRS and hence a CPSC metric.

As a corollary, any $3$-dimensional closed Type I $\kappa$-solution must be
isometric to a shrinking spherical space form. The reason is as follows. By
B.-L. Chen \cite{ChenStrongUniqueness}, $\operatorname{Rm}\geq0$. If
$\operatorname{Rm}>0$, then $g(\tau)$ is a CPSC metric by Ni's theorem. On the
other hand, if the sectional curvatures are not positive, then $\mathcal{M}%
^{3}$ is covered by $\mathcal{S}^{2}\times\mathbb{R}$. Since any closed such
solution is $\kappa$-collapsed, we are done.

Observe that, by Brendle and Schoen \cite{BrendleSchoen} and Brendle
\cite{BrendleHarnack} (the latter enabling Perelman's $\kappa$-solution theory
to extend), Ni's theorem holds under Brendle--Schoen postivity of curvature.

In this note we observe that the combined results of Perelman \cite{Perelman1}%
, Naber \cite{Naber}, Enders, M\"{u}ller and Topping \cite{EMT}, and Zhang and
the first author \cite{CaoZhang} yield the following special case of the
assertion by Perelman (private communication to Ni) that any $3$-dimensional
Type I $\kappa$-solution with PCO must be a shrinking CPSC metric. As we
mentioned in the abstract, this result is implied by the earlier work of
Ding \cite{Ding} and is generalized in the recent work of the third
author \cite{Zhang}, where the condition of being Type I forward in time
is removed.

\begin{proposition} \label{proposition}
Suppose that $(\mathcal{M}^{3},g(\tau))$, $\tau\in(0,\infty)$, is a $\kappa
$-solution to the backward Ricci flow with PCO forming a singularity at
$\tau=0$ and satisfying $\left\vert \operatorname{Rm}\right\vert (\tau
)\leq\frac{A}{\tau}$, then $\mathcal{M}$ is closed and $g(\tau)$ is a
shrinking CPSC metric.
\end{proposition}

Note that we have assumed that the solution is Type I both forward and
backward in time. Applications of this result to the study of shrinking
gradient Ricci soliton (GRS) singularity models follow from Naber
\cite[\S 5]{Naber}, Lu and the second author \cite[Theorem 3]{ChowLuSplit},
and Munteanu and Wang \cite{MunteanuWangAS}.

\section{Proof of the proposition}
Before we proceed to prove Proposition \ref{proposition}, we prove the following lemma that asserts the existence of a singular point at the forward singular time on a $3$-dimensional $\kappa$-solution. This is crucial in proving that the blow-up limit is nonflat. The existence of such a point is an issue because of the noncompactness of $\mathcal{M}$; see Remark 1.1 in \cite{EMT}. We actually prove that \emph{every} point of $\mathcal{M}$ is a singular point.

\begin{lemma}\label{lemma}
Let $(\mathcal{M}^{3},g(\tau))$, where $\tau\in(0,\infty)$, be a $\kappa$-solution that forms a singularity at $\tau=0$ in the sense that $\displaystyle \lim_{\tau\rightarrow0^+}\sup_{x\in \mathcal{M}}R(x,\tau)=\infty$, where $R$ denotes the scalar curvature. Then every $p\in \mathcal{M}$ is a singular point in the sense that $\displaystyle \limsup_{\tau\rightarrow 0^+}R(p,\tau)=\infty$.
\end{lemma}

\begin{proof}
Since $0$ is a singular time, by definition we may find a sequence $\{(x_i,\tau_i)\}_{i=1}^\infty$, such that $\tau_i\searrow 0$ and $R(x_i,\tau_i)\rightarrow\infty$. Suppose $p\in \mathcal{M}$ is not a singular point. Then there exists $C<\infty$ such that $R(p,\tau_i)\leq C$ for every $i\in\mathbb{N}$. By Hamilton's trace Harnack estimate \cite{Hamilton1993}, we have $\displaystyle \frac{\partial R}{\partial \tau}\leq 0$. Hence $R(p,\tau_i)\in[c,C]$, for all $i\in\mathbb{N}$, where we denote $c=R(p,\tau_1)>0$. Define $g_i(\tau)=g(\tau+\tau_i)$. Then we can use Perelman's $\kappa$-compactness theorem \cite{Perelman1} to extract a (not relabelled) subsequence from $\displaystyle \{(\mathcal{M},g_i(\tau),(p,0))_{\tau\in[0,\infty)}\}_{i=1}^\infty$, which converges to a $\kappa$-solution $(\mathcal{M}_\infty,g_\infty(\tau),(p_\infty,0))_{\tau\in[0,\infty)}$. In particular, $(\mathcal{M}_\infty,g_\infty(0))$ has bounded curvature. Let $A<\infty$ be the curvature bound of $(\mathcal{M}_\infty,g_\infty(0))$. By the definition of pointed smooth Cheeger--Gromov convergence and by passing to a suitable subsequence, there exists a sequence of open precompact sets $\{U_i\}_{i=1}^\infty$ exhausting $(\mathcal{M}_\infty,g_\infty(0))$, where each $U_i$ contains $p_\infty$, and there exists a sequence of diffeomorphisms
\begin{eqnarray*}
\psi_i:U_i&\rightarrow& V_i\subset(\mathcal{M},g_i(0)),
\\
\psi_i(p_\infty)&=&p,
\end{eqnarray*}
with the following properties. We have $\overline{B_{g_i(0)}(p,i)}\subset V_i$ and that $\psi_i^*g_i(0)$ is $i^{-1}$-close to $g_\infty(0)$ on $U_i$ with respect to the $C^i$-topology. Notice here that we actually have Cheeger--Gromov convergence of the solutions of the backward Ricci flow on the whole time interval $[0,\infty)$, but we need only to use the convergence on the time zero slice. Let $i_1\in\mathbb{N}$ be large enough so that $R(x_{i_1},\tau_{i_1})>100A$, where the existence of $i_1$ is guaranteed by the assumption that $R(x_i,\tau_i)\rightarrow\infty$. Then we select $i_2>i_1$ such that $dist_{g_{i_1}(0)}(p,x_{i_1})=dist_{g(\tau_{i_1})}(p,x_{i_1})<100^{-1}i_2$. Since the Ricci flow with nonnegative curvature shrinks distances forward in time, it follows that $dist_{g(\tau_{i_2})}(p,x_{i_1})<100^{-1}i_2$ and hence that $x_{i_1}\in\overline{B_{g_{i_2}(0)}(p,i_2)}\subset V_{i_2}$. Moreover, by Hamilton's trace Harnack estimate \cite{Hamilton1993} we have $R(g_{i_2}(0))(x_{i_1})=R(x_{i_1},\tau_{i_2})\geq R(x_{i_1},\tau_{i_1})>100A$, since $\tau_{i_2}<\tau_{i_1}$. This yields a contradiction when $i_2$ is large enough (say $i_2>10000$) since $\psi_{i_2}^{-1}(x_{i_1})$ is contained in the set $U_{i_2}$ on which $\psi_{i_2}^*g_{i_2}(0)$ is $i_2^{-1}$-close to $g_\infty(0)$ with respect to the $C^{i_2}$-topology, while the curvature of $g_\infty(0)$ is bounded by $A$.
\end{proof}

\bigskip
We now give the proof of our main result.
\bigskip

\begin{proof}[Proof of Proposition \ref{proposition}]
By Ni's theorem, we may suppose that $\mathcal{M}^{3}$ is noncompact, so that
$\mathcal{M}$ is diffeomorphic to $\mathbb{R}^{3}$. By the first part of
Theorem 3.1 in \cite{Naber}, for any $x\in\mathcal{M}$, $\tau_{i}%
^{-}\rightarrow0$, and $\tau_{i}^{+}\rightarrow\infty$, $(\mathcal{M}%
,(\tau_{i}^{\pm})^{-1}g(\tau_{i}^{\pm}\tau),(x,1))$ subconverges to a noncompact
shrinking GRS $(\mathcal{M}^{\pm},g^{\pm}(\tau),(x^{\pm},1))$ which does not contain any
embedded $\mathbb{R}P^{2}$. By Theorem 1.1 in \cite{EMT} and by Lemma \ref{lemma} above, $(\mathcal{M}%
^{-},g^{-}(\tau))$ is nonflat since every point is a singular point, whereas by Theorem 4.1 in \cite{CaoZhang} (see
also the statements in its proof), $(\mathcal{M}^{+},g^{+}(\tau))$ is nonflat,
since both of these results apply to noncompact manifolds. By Lemma 1.2 in
Perelman \cite{Perelman2}, $g^{\pm}(\tau)$ cannot have PCO. Thus the
$(\mathcal{M}^{\pm},g^{\pm}(\tau))$ are isometric to (shrinking) round cylinders
$\mathcal{S}^{2}\times\mathbb{R}$. By the second part of Theorem 3.1 in
\cite{Naber}, we conclude that the same is true for $(\mathcal{M}^{3}%
,g(\tau))$, which contradicts $g(\tau)$ having PCO.
\end{proof}

\begin{remark}
In \cite{Zhang} by the third author, it shown that there do not exist $3$-dimensional noncompact 
PCO $\kappa$-solutions only assuming the solution is Type I backward. This confirms an assertion
that Grisha Perelman made to Lei Ni.
\end{remark}

\section{A criterion for ancient solutions to form forward singularities}

In this section we present an application of Lemma \ref{lemma}, which gives a necessary and sufficient condition for a $3$-dimensional $\kappa$-solution to form a forward singularity.

\begin{corollary}
A $3$-dimensional $\kappa$-solution forms a forward singularity if and only if at some time slice $\inf_{\mathcal{M}} R>0$.
\end{corollary}

\begin{proof}
Let $(\mathcal{M}^{3},g(\tau))$, where $\tau\in(0,\infty)$, be a $\kappa$-solution to the backward Ricci flow that forms a singularity at $\tau=0$. By Lemma \ref{lemma}, for every $p\in \mathcal{M}$, $R(p,\tau)$ increases to infinity as $\tau\searrow0$. By integrating Perelman's derivative estimate \cite{Perelman1}
\begin{eqnarray*}
\left|\frac{\partial R}{\partial\tau}\right|\leq\eta R^2,
\end{eqnarray*}
where $\eta$ depends only on $\kappa$, from $0$ to $\tau$, we have
\begin{eqnarray*}
R(p,\tau)\geq\frac{1}{\eta\tau}
\end{eqnarray*}
for every $p\in \mathcal{M}$ and $\tau\in(0,\infty)$. It follows immediately that $\displaystyle\inf_{p\in \mathcal{M}}R(p,\tau)>0$ for every $\tau\in(0,\infty)$.
\\

On the other hand, suppose $(\mathcal{M}^{3},g(\tau))$, where $\tau\in[0,\infty)$, is a $\kappa$-solution to the backward Ricci flow such that $\inf_{p\in \mathcal{M}}R(p,T)=c>0$ for some $T>0$. We use an idea of Perelman \cite{Perelman2} to show that the solution cannot be extended forward to time infinity. Up to scaling the solution by a constant factor, we can find a sequence $x_i\rightarrow\infty$, such that $\displaystyle \lim_{i\rightarrow\infty}R(x_i,T)=1$. Applying the $\kappa$-compactness theorem \cite{Perelman1}, we can extract a (not relabelled) subsequence of $\{(\mathcal{M},g(\tau+T),(x_i,0))\}_{i=1}^\infty$, converging to a $\kappa$-solution $(\mathcal{M}_\infty,g_\infty(\tau),(x_\infty,0))$, which must be the shrinking cylinder since we have splitting at infinity; see \cite{Perelman1}. Moreover, we have $R_\infty(x_\infty,0)=1$ and $(\mathcal{M}_\infty,g_\infty(\tau))$ has unbounded curvature as $\tau\rightarrow-1$. Then we can conclude that $(\mathcal{M},g(\tau))$ becomes singular as $\tau\rightarrow T-1$. For suppose this is not the case. Then there exists an $\varepsilon>0$ such that $R(g(\tau))$ is uniformly bounded for $\tau\in[T-1-\varepsilon,\infty)$. It then follows that the limit flow $(\mathcal{M}_\infty,g_\infty(\tau))$ exists and has bounded curvature for $\tau\in[-1-\varepsilon,\infty)$, which is a contradiction.
\end{proof}

\begin{acknowledgement}
We would like to thank Peng Lu, Ovidiu Munteanu, Lei Ni, and Jiaping Wang for
helpful discussions. X. Cao's research was partially supported by a grant from the Simons Foundation (\#280161).
\end{acknowledgement}

\end{document}